\documentclass[11pt]{amsart}
\usepackage[vmargin=2.5cm,hmargin=2.5cm]{geometry}
\usepackage{enumerate}
\usepackage{color}
\usepackage{url}

\newtheorem{theorem}{Theorem}
\newtheorem{lemma}[theorem]{Lemma}
\newtheorem{corollary}[theorem]{Corollary}
\newtheorem{proposition}[theorem]{Proposition}

\newtheorem{defn}[theorem]{Definition}

\theoremstyle{remark}
\newtheorem{example}[theorem]{Example}
\newtheorem{remark}[theorem]{Remark}

\def\peb[#1]{{\left\lfloor #1\right\rfloor}}
\newcommand{\bs}{\boldsymbol}

\title{Delta sets for numerical semigroups with embedding dimension three} 

\author{P. A. Garc\'{\i}a-S\'{a}nchez}
\address{Departamento de \'Algebra and CITIC-UGR,
   Universidad de Granada, 18071 Granada, Espa\~na}
\email{pedro@ugr.es}

\thanks{The first author is supported by the projects FQM-343, FQM-5849, plan propio Universidad de Almer\'{\i}a and FEDER funds}

\author{D. Llena}
\address{Departamento de Matem\'aticas, Universidad de
   Almer\'{\i}a, 04120 Almer\'{\i}a, Espa\~na}
\email{dllena@ual.es}

\thanks{The second author is supported by the project FQM-343 and FEDER funds}

\author{A. Moscariello}
\address{Dipartimento di Matematica e Informatica, \ Universit\`a di Catania, \  Viale Andrea Doria 6, 
 and Scuola Superiore di Catania, \ Universit\`a di Catania, \  Via Valdisavoia 9, 
95125 Catania, Italy.}

\email{alessio.moscariello@studium.unict.it}


\begin{document}
\maketitle

\begin{abstract}
We present a fast algorithm to compute the Delta set of a nonsymmetric numerical semigroups with embedding dimension three. 
\end{abstract}

A monoid is a half-factorial monoid if for every element all the lengths of all the factorizations of this element in terms of atoms remain the same. Delta sets were introducee to measure how far a monoid can be from being half-factorial, and thus how wild the sets of lengths of factorizations are (\cite{G}). Geroldinger in \cite{G} presented the first results on Delta sets, also known as sets of distances, computing in particular the minimum distance between any two factorizations with consecutive lengths. It was shown in \cite{CGLMS} that for a monoid with bounded sets of lengths of factorizations, the maximum was reached in a particular class of elements, known as Betti elements (which are important for minimal presentation computations).

Recently Delta sets have been intensively studied on numerical semigroups (\cite{CHK,CKDH,CKLNZ}). It has been shown that Delta sets are eventually periodic (\cite{CHK}), and a bound for this periodicity was presented in that paper. As a byproduct, we get a procedure to compute the Delta set of a numerical semigroup (which is the union of all Delta sets of its elements). This bound was improved in \cite{BB}, then in \cite{GMV}, and lately in \cite{MOP}, where the fastest procedure to compute the Delta set of a numerical monoid is presented, based on dynamic programing. Christopher O'Neil implemented this procedure for the \texttt{GAP} (\cite{GAP}) package \texttt{numericalsgps} (\cite{numericalsgps}). In \cite{CKLNZ} it is shown that when the generators are too close to each other the Delta set of the numerical semigroup becomes the simplest possible: a singleton.

In the present manuscript we intend to understand better the behavior of Delta sets of element in a nonsymmetric numerical semigroup generated by three elements. As a consequence of this study we answer a question proposed by Scott Chapman during the International Meeting on Numerical Semigroups held in Vila Real on 2012. We are also able to compute the Delta set of these monoids with the same complexity as Euclid's greatest common divisor algorithm. We will show some examples of execution times comparing this new approach with the current implementation in \cite{numericalsgps} (which is meant for any numerical semigroup).

As it was pointed out in \cite{CGLMS}, minimal presentations are a fundamental tool to study Delta sets, and we take advantage that minimal presentations of nonsymmetric numerical semigroups with embedding dimension three are well known (\cite[Chapter 9]{numericalsgps}), and are ``unique''.

\section{Preliminaries}
Let $\mathbb N$ denote the set of nonnegative integers. Given $n_1,n_2,n_3\in\mathbb{N}$ with $\gcd(n_1,n_2,n_3)=1$, the \emph{numerical semigroup} generated by $\{n_1,n_2,n_3\}$ is the set $S=\langle n_1,n_2,n_3\rangle=\{x_1n_1+x_2n_2+x_3n_3 \mid (x_1,x_2,x_3)\in\mathbb{N}^3\}$, which is a submonoid of $(\mathbb N,+)$. We will assume that $n_1<n_2<n_3$, and that $\{n_1,n_2,n_3\}$ is a minimal generating system for $S$, that is, there is no $a,b\in \mathbb N$ such that $n_i=an_j+bn_k$ with $\{i,j,k\}=\{1,2,3\}$. In this setting it is said that $S$ is a semigroup with embedding dimension 3. 

The \emph{set of factorizations} of $s\in S$ is $\mathsf Z(s)=\{(x_1,x_2,x_3)\in \mathbb N^3 \mid x_1n_1+x_2n_2+x_3n_3=s\}$. We denote the \emph{length} of a factorization $\mathbf x=(x_1,x_2,x_3)\in \mathsf Z(s)$ as $|\mathbf x|=x_1+x_2+x_3$. We will use $|\mathbf x|=x_1+x_2+x_3$ for any $\mathbf x\in\mathbb Z^3$. The \emph{set of lengths} of $s\in S$ is $\mathcal L(s)=\{|\mathbf x|\mid \mathbf x\in\mathsf Z(s)\}$. It is easy to see that $\mathcal L(s) \subset [0,s]$, and consequently $\mathcal L(s)$ is finite. So it is of the form $\mathcal L(s)=\{m_1,\ldots ,m_k\}$ for some positive integers $m_1<m_2<\cdots <m_k$.  The set
\[
\Delta(s)=\{m_i-m_{i-1} \mid 2\le i\le k\}.
\]
is known as the \emph{Delta set} of $s \in S$, and 
the \emph{Delta set} of $S$ is 
\[\Delta(S)=\cup_{s \in S}\Delta(s).\]

As a particular instance of \cite[Lemma 3]{G}, we get the following result.

\begin{theorem}
Let $S$ be a numerical semigroup. Then 
\[
\min\Delta(M)=\gcd\Delta(M).
\]
Set $d=\gcd\Delta(M)$. There exists $k\in \mathbb N\setminus\{0\}$ such that
\[
\Delta(M)\subseteq\{d,2d,\ldots ,kd\}.
\]
\end{theorem}

Actually this $k$ is fully determined in our setting in \cite{CGLMS}.

The goal of this paper is to describe a fast procedure to compute this set. We start recalling some results and definitions.

Given $\{i,j,k\}=\{1,2,3\}$, define
\[
 c_i=\min\{k\in\mathbb{Z}^+ \mid kn_i \in\langle n_j,n_k \rangle\}.
\]

Then there exists $r_{ij},r_{ik}\in \mathbb N$ such that 
\[ c_in_i=r_{ij}n_j+r_{ik}n_k.\]

The condition $\gcd(n_1,n_2,n_3)=1$ is equivalent to $\#(\mathbb N\setminus S)<\infty$. Let $\mathrm F=\max(\mathbb Z\setminus S)$, the \emph{Frobenius number} of $S$. We say that $S$ is \emph{symmetric} if whenever $x\in \mathbb Z\setminus S$, then $\mathrm F-x\in S$. 

\begin{proposition}[\protect{\cite[Theorem 3]{J}}]\label{cr}
If $S$ is not symmetric, then the $r_{ij},r_{ik} \in \mathbb{Z}^+$ are unique. Moreover, $c_i = r_{ji}+r_{ki}$.
\end{proposition}

From $n_1<n_2<n_3$ we obtain the following result.

\begin{lemma}
Under the standing hypothesis, $c_1 > r_{12}+r_{13}$ and $c_3 < r_{31}+r_{32}$.
\end{lemma}

Set 
\[\delta_i=|c_i-r_{ij}-r_{ik}|\] 
for every $\{i,j,k\}=\{1,2,3\}$. By the previous lemma $\delta_1=c_1-r_{12}-r_{13}$ and $\delta_3=r_{31}+r_{32}-c_3$.
Also, from Proposition \ref{cr}, $\delta_2=|\delta_1-\delta_3|$.

\begin{lemma}[\protect{\cite[Corollary 3.1]{CGLMS}}]\label{minmax}
Under the standing hypothesis, $\min \Delta(S)=\gcd(\delta_1,\delta_3)$ and $\max \Delta(S) = \max\{\delta_1,\delta_3\}$.
\end{lemma}

\begin{remark}
In light of the last lemma, we can consider $\delta_1\neq\delta_3$ because in other case we will have $\min \Delta(S)=\max \Delta(S)= \delta_1=\delta_3$. And then $\Delta(S)=\{\delta_1\}$.
\end{remark}

Define $\varphi\colon\mathbb N^3\rightarrow S$ as $\varphi(x_1,x_2,x_3)=x_1n_1+x_2n_2+x_3n_3$. Then $\varphi$ is a monoid epimorphism and thus $S\cong \mathbb N^3/\ker \varphi$, where $\ker \varphi=\{(\mathbf a,\mathbf b)\in \mathbb N^3\times \mathbb N^3\mid \varphi(\mathbf a)=\varphi(\mathbf b)\}$. Associated to $\ker \varphi$, we define the subgroup $M=\{\mathbf a-\mathbf b \mid (\mathbf a,\mathbf b)\in \ker \varphi\}$ of $\mathbb Z^3$. Notice that if $s\in S$ and $\mathbf x,\mathbf y\in \mathsf Z(s)$, then $\mathbf x-\mathbf y\in M$.

A \emph{presentation} of $S$ is a system of generators of the congruence $\ker \varphi$. It is well known (see for instance \cite[Example 8.23]{RG}) that 
\[ 
\sigma=\{((c_1,0,0),(0,r_{12},r_{13})),((0,c_2,0),(r_{21},0,r_{23})),((0,0,c_3),(r_{31},r_{32},0))\}
\]
is a (minimal) presentation of $S$. It follows easily that if we set 
\[\mathbf v_1= (c_1,-r_{12},-r_{13}),\ \mathbf v_2=(-r_{21},c_2,-r_{23}) \hbox{ and }\mathbf v_3=(r_{31},r_{32},-c_3),\] 
then $M$ is generated as a group by $\{\mathbf v_1,\mathbf v_2,\mathbf v_3\}$. In light of Proposition \ref{cr}, $\mathbf v_2=\mathbf v_3-\mathbf v_1$, and consequently we obtain the following result.  

\begin{proposition}\label{fact0}
Let $s \in S$ and $\mathbf x,\mathbf y\in \mathsf{Z}(s)$. Then there exists $\lambda_1,\lambda_3\in \mathbb Z$ such that 
$
\mathbf x-\mathbf y=\lambda_1\mathbf v_1+\lambda_3 \mathbf v_3$.

\end{proposition}

\section{B\'ezout Couples}
A natural way to study $\Delta(S)$ passes through a better understanding of $M$. This is because $\delta\in\Delta(S)$ if and only if 
\begin{enumerate}
\item there exists $\mathbf x,\mathbf y\in \mathsf Z(s)$ for some $s\in S$, such that $|\mathbf x|>|\mathbf y|$ and $\delta=|\mathbf x|-|\mathbf y|$ ($=|\mathbf x-\mathbf y|$), and
\item there is no $\mathbf z\in \mathsf Z(s)$ such that $|\mathbf x|>|\mathbf z|>|\mathbf y|$.
\end{enumerate}
The first condition relies on $M$ and for the second we introduce the concept of B\'ezout couples.

\begin{proposition} Let $\delta_1,\delta_3\in \mathbb Z^+$ and $g=\gcd(\delta_1,\delta_3)$. Then for every $i\in \mathbb Z^+$,
\begin{itemize}
\item there exists a unique couple $(\lambda_{i1},\lambda_{i3}) \in \mathbb{Z} \times \mathbb{Z}$ such that $\lambda_{i1}\frac{\delta_1}g + \lambda_{i3}\frac{\delta_3}g =  i$ and $0 < \lambda_{i3} \le  \frac{\delta_1}g$,
\item there exists a unique couple $(\mu_{i1},\mu_{i3}) \in \mathbb{Z} \times \mathbb{Z}$ such that $\mu_{i1}\frac{\delta_1}g + \mu_{i3}\frac{\delta_3}g =  i$ and $0 < \mu_{i1} \le \frac{\delta_3}g$.
\end{itemize}
\end{proposition}
\begin{proof}
Follows from elementary number theoretic arguments.
\end{proof}
From now on, we will assume that $\gcd(\delta_1,\delta_3)=1$, otherwise we normalize $\delta_1$ and $\delta_3$ by $\gcd(\delta_1,\delta_3)$ as in the statement of the last proposition. 

\begin{defn}\label{couples}
Let $\delta_1,\delta_3 \in \mathbb{Z}^+$ be such that $\gcd(\delta_1,\delta_3)=1$, and let $i \in \{1,\ldots,\max\{\delta_1,\delta_3\}\}$.
\begin{enumerate}
\item Define the $\lambda$-B\'ezout couple of $i \in \mathbb{Z}^+$ as the unique couple $(\lambda_{i1},\lambda_{i3}) \in \mathbb{Z} \times \mathbb{Z}$ such that $\lambda_{i1}\delta_1 + \lambda_{i3}\delta_3 = i$ and $0 < \lambda_{i3} \le \delta_1$. We will denote this by $\bs \lambda_i =(\lambda_{i1},\lambda_{i3})$.
\item Define the {$\mu$-B\'ezout couple} of $i \in \mathbb{Z}^+$ as the unique couple $(\mu_{i1},\mu_{i3}) \in \mathbb{Z} \times \mathbb{Z}$ such that $\mu_{i1}\delta_1 + \mu_{i3}\delta_3 = i$ and $0 < \mu_{i1} \le \delta_3$. We will write $\bs \mu_i = (\mu_{i1},\mu_{i3})$.
\end{enumerate}
Set 
\[\mathcal B^{(\bs \lambda)}_{\delta_1,\delta_3}=\big\{\bs \lambda_i \mid i\in \{1,\ldots,\max\{\delta_1,\delta_3\}\}\big\} \quad \hbox{and}\quad 
\mathcal B^{(\bs \mu)}_{\delta_1,\delta_3}=\{\bs \mu_i \mid i\in \{1,\ldots,\max\{\delta_1,\delta_3\}\}\big\}.\]
We will say that a pair is a B\'ezout couple if it is either a $\lambda$-B\'ezout or a $\mu$-B\'ezout couple.
\end{defn}

We will associate to some particular B\'ezout couples possible values in the Delta set of $S$. For this reason, in light of Lemma \ref{minmax}, in the previous Definition we are only interested in the case $i\in\{1,\ldots,\max\{\delta_1,\delta_3\}\}$. Now we give some properties of B\'ezout couples:

\begin{lemma}\label{negativo-l1}
If $1\le i \le\max\{\delta_1,\delta_3\}$,we have
\begin{enumerate}
\item $-\delta_3<\lambda_{i1}\le 0$ and $-\delta_1< \mu_{i3}\le 0$, and 
\item $\lambda_{i1}+\delta_3=\mu_{i1}$ and $\lambda_{i3}-\delta_1=\mu_{i3}$.
\end{enumerate}
\end{lemma}
\begin{proof}
\begin{enumerate}[(1)]
\item As $i\ge 1$, $-\lambda_{i1}\delta_1<\lambda_{i3}\delta_3 \le \delta_1\delta_3$. So $-\lambda_{i1}<\delta_3$. Similarly we have $-\mu_{i3}<\delta_1$. Since $i\le \max\{\delta_1,\delta_3\}$, we obtain:
\begin{itemize}
\item $\lambda_{i1}\delta_1=i-\lambda_{i3}\delta_3<i\le \delta_1$, if $\delta_3<\delta_1$;
\item $\lambda_{i1}\delta_1=i-\lambda_{i3}\delta_3\le 0$ if $\delta_1<\delta_3$.
\end{itemize} In both cases, we obtain $\lambda_{i1}\le 0$. Similarly $\mu_{i3}\le 0$.
\item Subtracting both expressions of $i$ from the Definition \ref{couples} we obtain:
\[
(\lambda_{i1}-\mu_{i1})\delta_1 + (\lambda_{i3}-\mu_{i3})\delta_3 = 0.
\]
And, as $\gcd(\delta_1,\delta_3)=1$ we have that there exists $a\in \mathbb Z$ such that $\lambda_{i1}-\mu_{i1}=a\delta_3$ and $\lambda_{i3}-\mu_{i3}=-a\delta_1$. We know that $0<\lambda_{i3}-\mu_{i3}<\delta_1+\delta_1=2\delta_1$, whence we have that $a=-1$.\qedhere
\end{enumerate}
\end{proof}

\begin{defn}\label{irr}
Let $\bs \lambda_i$ be the $\lambda$-B\'ezout couple of $i \in \{1,\ldots ,\max\{\delta_1,\delta_3\}\}$. We say that $\bs \lambda_i$ is irreducible if there is no $j,k \in \{1,\ldots ,\max\{\delta_1,\delta_3\}\}$ such that $\bs \lambda_i=\bs \lambda_j+\bs \lambda_k$.
Similarly, let $\bs \mu_i$ be the $\mu$-B\'ezout couple of $i \in \{1,\ldots ,\max\{\delta_1,\delta_3\}\}$. We say that $\bs \mu_i$ is {irreducible} if there is no $j,k \in \{1,\ldots ,\max\{\delta_1,\delta_3\}\}$ such that $\bs \mu_i=\bs \mu_j+\bs \mu_k$.
\end{defn}

\begin{remark}\label{withone}
 All $\lambda$-B\'ezout couples of the form $(x,1)$ and all $\mu$-B\'ezout couples of the form $(1,y)$ are irreducible Bezout couples.
\end{remark}

\begin{lemma}\label{romper}
If $\bs x_i =\bs x_j+\bs x_k$, with $\bs x\in \{\bs \lambda, \bs \mu\}$, then $i=j+k$. 
\end{lemma}
\begin{proof}
The proof follows from the definition.
\end{proof}

We will denote 
\[\mathcal I^{(\bs\lambda)}_{\delta_1,\delta_3}=\left\{\bs \lambda_i\in\mathcal B^{(\bs\lambda)}_{\delta_1,\delta_3} ~\middle|~ \bs \lambda_i\mbox{ irreducible}\right\}\quad \mbox{and}\quad \mathcal I^{(\bs\mu)}_{\delta_1,\delta_3}=\left\{\bs \mu_i\in\mathcal B^{(\bs\mu)}_{\delta_1,\delta_3} ~\middle|~ \bs \mu_i\mbox{ irreducible}\right\}.\]
Next we translate the concept of irreducibility to the subgroup $M$.

Given $\mathbf z=(z_1,z_2,z_3)\in \mathbb Z^3$ we can always write $\mathbf z=\mathbf z^+-\mathbf z^-$ with $\mathbf z^+,\mathbf z^-\in \mathbb N^3$ and $\mathbf z^+\cdot \mathbf z^-=0$ (dot product).

\begin{lemma}\label{fact}
Let $s \in S$ and $\mathbf x=(x_1,x_2,x_3) \in \mathsf Z(s)$. Let $\bs\alpha =(\alpha_1,\alpha_2,\alpha_3)\in M$. Then $\mathbf x+\bs\alpha \in \mathsf Z(s)$ if and only if $\mathbf x \ge \bs\alpha^-$.
\end{lemma}
\begin{proof}
Obviously $\mathbf x + \bs\alpha \in \mathsf Z(s)$ if and only if $x_j+\alpha_j\ge 0$ for $j\in\{1,2,3\}$, and this happens if and only if $x_j \ge -\alpha_j$ whenever $\alpha_j \le 0$. This is equivalent to $x_j \ge \alpha^-_j$. 
\end{proof}
\begin{defn}\label{tau}
Let $\bs x_i=(x_{i1},x_{i3})\subseteq \mathbb Z^2$ be such that $x_{i1}\delta_1+x_{i3}\delta_3=i$, for some  $i\in\{1,\ldots,  \max\{\delta_1,\delta_3\}\}$. Denote by $\tau_{\bs x_i}$ the vector 
\begin{eqnarray*}
\mathbf{\tau}_{\bs x_i}=(\tau_{i1},\tau_{i2},\tau_{i3}):= x_{i1}\mathbf v_1+x_{i3}\mathbf v_3\in M.
\end{eqnarray*}
\end{defn}

\begin{remark}\label{mainremark}
Observe that $|\mathbf{\tau}_{\bs x_i}|=x_{i1}|\mathbf v_1|+x_{13}|\mathbf v_3|=x_{i1}\delta_1+x_{13}\delta_3=i$.
\end{remark}

\begin{lemma}\label{secondcoordinate}
Let $\bs x_i=(x_{i1},x_{i3})\in \mathbb Z^2$ be such that $x_{i1}\delta_1+x_{i3}\delta_3=i$, with $i\in\{1,\dots, \max\{\delta_1,\delta_3\}\}$. Then $x_{i1}\le 0$ if and only if $\tau_{i2}>0$.
\end{lemma}
\begin{proof}
Since $x_{i1}\delta_1+x_{i3}\delta_3=i$, the condition $x_{i1}\le 0$ forces $x_{i3}>0$. So $-x_{i1}r_{12}+x_{i3}r_{32}>0$. Analogously,  $x_{i1}> 0$ implies $x_{i3}\le 0$, whence $-x_{i1}r_{12}+x_{i3}r_{32}<0$.
\end{proof}

\begin{lemma}\label{twisted}
Let $\bs x_i=(x_{i1},x_{i3})\in \mathbb Z^2$ be such that $x_{i1}\delta_1+x_{i3}\delta_3=i$, with $i\in\{1,\dots, \max\{\delta_1,\delta_3\}\}$. If   $\delta_j< \delta_k$, with $\{j,k\}=\{1,3\}$, then $|x_{ik}|\le |x_{ij}|$.
\end{lemma}
\begin{proof}
We have that $0<x_{ij}\delta_j+x_{ik}\delta_k=i\le\delta_k$. We can divide by $\delta_k$ to obtain: $0<x_{ij}\frac{\delta_j}{\delta_k}+x_{ik}\le 1$.
If $x_{ik}\le 0$, we have $0\le |x_{ik}|=-x_{ik}<x_{ij}\frac{\delta_j}{\delta_k}<x_{ij}=|x_{ij}|$. While if $x_{ik}> 0$, then $x_{ij}< \frac{\delta_j}{\delta_k}x_{ij}\le 1-x_{ik}\le 0$. So $x_{ij}\le -x_{ik}=-|x_{ik}|$.
\end{proof}

\begin{corollary}\label{performance}
Let $\bs x_i=(x_{i1},x_{i3})\in \mathbb Z^2$ be such that $x_{i1}\delta_1+x_{i3}\delta_3=i$, with $i\in\{1,\dots, \max\{\delta_1,\delta_3\}\}$. The following table describes the signs of both $x_{i3}$ and the coordinates of $\tau_{\bs x_i}$.
 \begin{center}
 \begin{tabular}{|l|c|c|c|c|c|r|}
 \hline
  & $x_{i1}\le 0$  & $x_{i1}>0$ \\
 \hline
 $\delta_1 > \delta_3$ & $\tau_{i2} > 0, \tau_{i3} < 0, x_{i3}>0$ & $\tau_{i2} < 0, \tau_{i3} > 0, x_{i3}\le 0$  \\
 \hline
 $\delta_3 > \delta_1$ & $\tau_{i2} > 0, \tau_{i1} < 0, x_{i3}>0$ & $\tau_{i2} < 0, \tau_{i1} > 0, x_{i3}\le 0$ \\
 \hline
  \end{tabular}
 \end{center}
\end{corollary}
\begin{proof}
  If $x_{i1}\le 0$ we have $x_{i3}>0$. 
  \begin{description}
  \item[$\delta_1>\delta_3$] from Lemma \ref{twisted}, we have $0\le -x_{i1}\le x_{i3}$; as $c_3>r_{31}$, we obtain $\tau_{i3}=-x_{i1}r_{13}-x_{i3}c_3<0$.
  \item[$\delta_1<\delta_3$] in view of Lemma \ref{twisted} again, we have that $0\le x_{i3}\le-x_{i1}$; from $c_1>r_{13}$, we obtain $\tau_{i1}<0$.
  \end{description} 
  Lemma \ref{secondcoordinate} states that $\tau_{i2}>0$.
   
  The case $x_{i1}>0$ is analogous.
\end{proof}

\begin{lemma}\label{fact2}
Let $\bs x_i$ be a reducible B\'ezout couple, with $\bs x\in \{\bs \lambda,\bs \mu\}$ and $1\le i\le \max\{\delta_1,\delta_3\}$. Then exists $l\in\mathbb Z^+$, $l<i$, such that $\tau_{\bs x_l}^- \le \tau_{\bs x_i}^-$ . 
\end{lemma}
\begin{proof}
Assume that $\bs x_i=\bs \lambda_i$ (the other case is analogous). As $\bs\lambda_i$ is reducible, it follows that there exist positive integers $j$ and $k$ such that $\bs \lambda_i =\bs \lambda_j + \bs \lambda_k$. Lemma \ref{romper} ensures that $k,j<i$, and it is easy to derive that 
\begin{equation}\label{sumavector}
\tau_{\bs \lambda_i}= \tau_{\bs \lambda_j}+\tau_{\bs \lambda_k}.
\end{equation}
From Lemma \ref{secondcoordinate}, $\tau_{i2}>0$, $\tau_{j2}>0$ and $\tau_{k2}>0$, so $\tau_{i2}^-=\tau_{j2}^-=\tau_{k2}^-=0$. For the other coordinates, we distinguish two cases.
\begin{enumerate}[(i)]
\item $\delta_1 > \delta_3$. From Corollary \ref{performance} we have $\tau_{i3}<0$, $\tau_{j3}<0$ and $\tau_{k3}<0$. Observe that $\lambda_{i3}\neq 0$ because otherwise $|\lambda_{i1}|\le |\lambda_{i3}|=0$, and thus $i=0$. So, as $\tau_{i3}=\tau_{j3}+\tau_{k3}$, and they are all negative, we deduce $\tau_{i3}^-=\tau_{j3}^-+\tau_{k3}^-$. Hence $\tau_{i3}^-\ge \tau_{j3}^-$ and $\tau_{i3}^-\ge\tau_{k3}^-$.

From the equation \eqref{sumavector} we have that $\tau_{i1}=\tau_{j1}+\tau_{k1}$. Thus, if $\tau_{i1},\tau_{j1},\tau_{k1}$ have all the same sign we can deduce, as before, $\tau^-_{j1}\le\tau^-_{i1}$ and $\tau^-_{k1}\le\tau^-_{i1}$, and in this case we can take $l=j$ or $l=k$ to finish the proof. While in the other case one and only one between $\tau_{j1},\tau_{k1}$ should be nonnegative (in other cases both nonpositive or both nonnegative implies that $\tau_{i1},\tau_{j1},\tau_{k1}$ have the same sign). We call it $\tau_{l1}$. So, for this $l$ we have $\tau^-_{l1}=0\le\tau^-_{i1}$, $0=\tau^-_{l2} \le \tau^-_{i2}=0$ and $\tau^-_{l3} \le \tau^-_{i3}$.

\item $\delta_1 < \delta_3$. Again from Corollary \ref{performance},  $\tau_{i1} < 0$,  $\tau_{j1}<0$ and $\tau_{k1}<0$. Hence from $\tau_{i1}=\tau_{j1}+\tau_{k1}$ we deduce $\tau_{i1}^-=\tau_{j1}^-+\tau_{k1}^-$. This leads to $\tau_{i1}^-\ge \tau_{j1}^-$ and $\tau_{i1}^-\ge \tau_{k1}^-$. 

Now arguing as above, but with $\tau_{i3}=\tau_{j3}+\tau_{k3}$, we have again two possibilities. 
\begin{enumerate}
\item The integers $\tau_{i3},\tau_{j3},\tau_{k3}$ have the same sign. We can take $l=j$ or $l=k$ to finish the proof. 
\item One of the $\tau_{j3},\tau_{k3}$ must be nonnegative. We choose it to conclude the proof.\qedhere
\end{enumerate}
\end{enumerate}
\end{proof}
\begin{lemma}\label{fact3}
With the same notation as the above lemma we have that exists $l\in\mathbb Z^+$, $l<i$, such that $\tau^+_{\bs x_l} \le \tau^+_{\bs x_i}$ .
\end{lemma}
\begin{proof}
Analogous to the preceding lemma.
\end{proof}
We denote 
\[I_{\delta_1,\delta_3}=I^{(\bs\lambda)}_{\delta_1,\delta_3}\cup I^{(\bs\mu)}_{\delta_1,\delta_3}.\]

\begin{theorem}\label{main}
Let $S$ be a nonsymmetric numerical semigroup minimally generated by $\{n_1,n_2,n_3\}$ with $n_1<n_2<n_3$. Let $\delta_1$, $\delta_2$ and $I_{\delta_1,\delta_2}$ be defined as above. Let $g=\gcd(\delta_1,\delta_3)$. Then
\[
\Delta(S)=\left\{g i \in\mathbb N \mid  i\in\{1, \ldots, \max\{\delta_1/g,\delta_3/g\}, \bs x_i \in I_{\delta_1/g,\delta_3/g}\right\}.
\]
\end{theorem} 

\begin{proof}
For sake of simplicity assume that $g=1$.

\noindent $\subseteq$. Let $m \in \Delta(S)$. It is obvious that $1 \le m \le \max\{\delta_1,\delta_3\}$.

Assume that $\bs\lambda_m  \not \in I^{(\bs\lambda)}_{\delta_1,\delta_3}$ and $\bs\mu_m \not \in I^{(\bs\mu)}_{\delta_1,\delta_3}$. Since $m \in \Delta(S)$ there exists $s \in S$, $\mathbf x=(x_1,x_2,x_3),\mathbf y=(y_1,y_2,y_3) \in \mathsf Z(s)$ such that $|\mathbf y|-|\mathbf x|=m$ and there is no $\mathbf z \in \mathsf Z(s)$ such that $|\mathbf x|<|\mathbf z| < |\mathbf y|$. By Proposition \ref{fact0} we have
\[
\mathbf y-\mathbf x=a_1 (c_1,-r_{12},-r_{13})+a_3 (r_{31},r_{32},-c_3)=a_1\mathbf v_1+a_3\mathbf v_3
\]
for some $a_1,a_3 \in \mathbb{Z}$. By taking lengths, we have $a_1\delta_1 + a_3\delta_3=m$. As $1\le m\le \max\{\delta_1,\delta_3\}$, we  deduce that $a_1a_3<0$. Also, $(a_1+q\delta_3)\delta_1+(a_3-q\delta_1)\delta_3=m$ for all $q\in \mathbb Z$. 

For sake of simplicity, write $\mathbf v=-\delta_3\mathbf v_1+\delta_1\mathbf v_3=(v_1,v_2,v_3)$.  As in Lemmas \ref{secondcoordinate} and \ref{twisted}, we can deduce:

\begin{itemize}
\item $v_2=\delta_3r_{12}+\delta_1r_{32}>0$.
\item $v_1=-\delta_3c_1+\delta_1r_{31}$; thus  if $\delta_3>\delta_1$, then $v_1<0$.
\item $v_3=\delta_3r_{13}-\delta_1c_3$; whence if $\delta_3<\delta_1$, then $v_3<0$.
\end{itemize}

\begin{enumerate}[(a)]
\item  $a_1 \le 0$, then $a_3>0$ and notice that we can take $q\in \mathbb N$ such that $(a_1+q\delta_3,a_3-q\delta_1)$ is a $\lambda$-B\'ezout couple. Write
\[
\mathbf y-\mathbf x= (a_1+q\delta_3)\mathbf v_1+(a_3-q\delta_1) \mathbf v_3
 -q\delta_3\mathbf v_1+q\delta_1\mathbf v_3 
=\tau_{\bs\lambda_m}+q\mathbf v.
\]
We going to prove that either $\mathbf y-\tau_{\bs\lambda_m}$ or $\mathbf x+\tau_{\bs\lambda_m}$ are in $\mathsf Z(s)$. From Lemma \ref{fact}, it suffices to show that either $\mathbf y>\tau_{\bs\lambda_m}^+$ or $\mathbf x>\tau_{\bs\lambda_m}^-$ (observe that $(-\tau_{\bs \lambda_m})^-=\tau_{\bs\lambda_m}^+$).  We have, from Corollary \ref{performance} and the above remark on $(v_1,v_2,v_3)$ that $\tau_{m2}>0$, $y_2-x_2>0$ and $v_2> 0$. Also $y_2-x_2=qv_2+\tau_{m2}$ with $q\ge 0$. So we can deduce $y_2>qv_2+\tau_{m2}\ge \tau_{m2}>0$. Now depending on $\delta_1<\delta_3$ or $\delta_3<\delta_1$, we can assure, again from Corollary \ref{performance} and the above remark that for $i=1$ or $i=3$ we have $\tau_{mi}<0$, $y_i-x_i<0$ and $v_i<0$. Hence $-x_i<y_i-x_i=qv_i+\tau_{mi}<\tau_{mi}<0$.

We have in this case ($a_1\le 0$) that $0<\tau_{m2}<y_2$ and $x_i<\tau_{mi}<0$. Take $j$ such that $\{i,j\}=\{1,3\}$. Now, if $\tau_{mj}\le 0$, we have $\tau_{\bs\lambda_m}^+=(0,\tau_{m2},0)<\mathbf y$, and if $\tau_{mj}>0$, then $\tau_{\bs\lambda_m}^-=-\tau_{mi}\bs e_i<(x_1,x_2,x_3)=\mathbf x$; where $\bs e_i$ is the $i$th row of the $3\times 3$ identity matrix.
\begin{itemize}
\item If $\mathbf y>\tau_{\bs\lambda_m}^+$, by Lemma \ref{fact} we have $\mathbf y-\tau_{\bs\lambda_m}\in \mathsf Z(s)$. As $\bs \lambda_m$ is reducible, by Lemma \ref{fact3}, there exists $j \in \mathbb{Z}^+,j<m$ such that $\tau^+_{\bs\lambda_m} \ge \tau^+_{\bs \lambda_j}=(-\tau_{\bs \lambda_j})^-$. As $\mathbf y>\tau_{\bs\lambda_m}^+$, in light of Lemma \ref{fact}, we have that
$\mathbf z=\mathbf y-\tau_{\bs \lambda_j}\in\mathsf Z(s)$.

\item If $\tau_{\bs\lambda_m}^-<\mathbf x$, again by Lemma \ref{fact} we deduce $\mathbf x+\tau_{\bs \lambda_j}\in\mathsf Z(s)$. By Lemmas \ref{fact2} and \ref{fact}, we derive $\mathbf z=\mathbf x+\tau_{\bs \lambda_j}\in\mathsf Z(s)$.
\end{itemize}
In both cases $|\mathbf x|<|\mathbf z|<|\mathbf y|$, which is a contradiction.

\item $a_1 > 0$. This case is identical, considering now $q$ a nonpositive integer such that $(a_1+q\delta_3,a_3-q\delta_1)$ is a $\mu$-B\'ezout couple. 
\end{enumerate}

\noindent $\supseteq$. Let $m \in \{i \in\mathbb N \mid 1 \le i \le \max\{\delta_1,\delta_3\}, \bs x_i \in I_{\delta_1,\delta_3}\}$. Thus we have $\bs x_m\in I_{\delta_1,\delta_3}$. Assume to the contrary that $m \not \in \Delta(S)$.

We know that $\tau_{\bs x_m}^-$ and $\tau_{\bs x_m}^+$ are factorizations for some $s\in S$ and, from Remark \ref{mainremark}, we have $|\tau_{\bs x_m}|=|\tau_{\bs x_m}^+|-|\tau_{\bs x_m}^-|=m$. Hence $|\tau_{\bs x_m}^+|=|\tau_{\bs x_m}^-|+m>|\tau_{\bs x_m}^-|$. Since $m \not \in \Delta(S)$, there exists some $\mathbf z \in \mathsf Z(s)$ such that  $|\tau_{\bs x_m}^-| < |\mathbf z|<| \tau^+_{\bs x_m}|$. 
By Proposition \ref{fact0}, we know that there exists $(a_1,a_3),(b_1,b_3)\in \mathbb Z^2$ such that
\begin{eqnarray*}
\mathbf z-\tau_{\bs x_m}^-=a_1\mathbf v_1 +a_3\mathbf v_3, \quad 
\tau_{\bs x_m}^+-\mathbf z=b_1\mathbf v_1+b_3\mathbf v_3,
\end{eqnarray*} 
and consequently
\[
0<a_1\delta_1 +a_3\delta_3<m \hbox{ and }  0<b_1\delta_1 +b_3\delta_3<m.
\]
Observe that $\mathbf z-\tau_{\bs x_m}^-, \tau_{\bs x_m}^+-\mathbf z\not\in\mathbb N^3$, since any two factorizations of the same element are incomparable.

Notice that, since $\{\bs v_1,\bs v_3\}$ is a basis of $M$,  
\begin{equation}\label{xab}
\bs x_m=(x_{m1},x_{m3})=(a_1,a_3)+(b_1,b_3).
\end{equation} 

We will prove in Lemma \ref{mainlemma} that both $(a_1,a_3)$ and $(b_1,b_3)$ are B\'ezout couples. Moreover, this lemma states that if $\bs x_m$ is a $\lambda$-B\'ezout couple, then $(a_1,a_3)$ is a $\lambda$-B\'ezout couple and $(b_1,b_3)$ is a B\'ezout couple. And, if $\bs x_m$ is a $\mu$-B\'ezout couple, then $(b_1,b_3)$ is a $\mu$-B\'ezout couple and $(a_1,a_3)$ is a B\'ezout couple.

Assume, then, that  $\bs x_m=\bs\lambda_m$.  It follows from the above mentioned Lemma \ref{mainlemma}, that $(a_1,a_3)$ is a $\lambda$-B\'ezout couple and $(b_1,b_3)$ is a B\'ezout couple. It is clear that $(b_1,b_3)$ can not be a $\lambda$-B\'ezout couple because $\bs x_m=\bs \lambda_m$ is irreducible. Hence in this setting $(b_1,b_3)$ is a $\mu$-B\'ezout couple. As $\bs x_m=\bs\lambda_m$, from Lemma \ref{secondcoordinate} we have $\tau_{m2}>0$, so if we write $\mathbf z=(z_1,z_2,z_3)$, we have
\[
\begin{array}{l}
\tau_{\bs\lambda_m}=(\tau_{m1},\tau_{m2},\tau_{m3}), \\ \mathbf z-\tau_{\bs \lambda_m}^-=(z_1-\tau_{m1}^-,z_2,z_3-\tau_{m3}^-), \\ \tau_{\bs \lambda_m}^+-\mathbf z=(\tau_{m1}^+-z_1,\tau_{m2}-z_2,\tau_{m3}^+-z_3).
\end{array}
\]

It follows that
\[
(\tau_{m1},\tau_{m2},\tau_{m3})=(z_1-\tau_{m1}^-,z_2,z_3-\tau_{m3}^-)+(\tau_{m1}^+-z_1,\tau_{m2}-z_2,\tau_{m3}^+-z_3).
\]
If we apply Corollary \ref{performance} to $(a_1,a_3)$ and $(b_1,b_3)$, we deduce the following.
\begin{enumerate}[(1)]
 \item If $\delta_1<\delta_3$, then $\tau_{m1}<0$ and  $\tau_{m1}^+-z_1>0$, so we have $z_1<\tau_{m1}^+=0$, contradicting that $\mathbf z\in \mathsf Z(s)\subseteq \mathbb N^3$.
 \item If $\delta_3<\delta_1$, then $\tau_{m3}<0$ and $\tau_{m3}^+-z_3>0$, so we have $z_3<\tau_{m3}^+=0$, which yields again a contradiction.   
\end{enumerate}

The case $\bs x_m=\bs\mu_m$ is analogous.
\end{proof}

In order to make the proof of Theorem \ref{main}, we have extracted Lemma \ref{mainlemma} from it. We need an extra lemma to prove this piece.

\begin{lemma}\label{technicallemma}
Let $m\le \max\{\delta_1,\delta_3\}$. Let $(x_{m1},x_{m3})$ be a B\'ezout couple such that $(x_{m1},x_{m3})=(a_1,a_3)+(b_1,b_3)$ with $0<a_1\delta_1+a_3\delta_3<m$ and $0<b_1\delta_1+b_3\delta_3<m$. 
\begin{enumerate}[(1)]
 \item If $a_3\le -\delta_1$, then $a_1>\delta_3$. Moreover, if $\delta_3>\delta_1$, the converse is also true.
 \item If $a_1\le -\delta_3$, then $a_3>\delta_1$. Moreover, if $\delta_1>\delta_3$, the converse holds.
 \item If ${\bs x_m}=\bs \lambda_m$, we have that $a_3\le -\delta_1$ implies $\delta_1< b_3$; while $\delta_1< b_3$ implies $a_3<0$. 
 \item If ${\bs x_m}=\bs \lambda_m$, then the inequality $a_1> \delta_3$ implies $b_1 < -\delta_3$; while $b_1 \le -\delta_3$ implies $0<a_1$.
 \item If ${\bs x_m}=\bs \mu_m$, then $a_1\le -\delta_3$ implies $\delta_3 < b_1$; while $\delta_3 < b_1$ implies $a_1<0$.
 \item If ${\bs x_m}=\bs \mu_m$, we have that $a_3 > \delta_1$ implies $b_3 < -\delta_1$; and $b_3 \le -\delta_1$ implies $0<a_3$.
\end{enumerate}
Clearly, the above statements are true if we swap $a_i$ and $b_i$.
\end{lemma}
\begin{proof}
\begin{enumerate}[(1)]
\item As $0<a_1\delta_1+a_3\delta_3$, if $a_3\le -\delta_1$,  then $0<a_1\delta_1+a_3\delta_3\le a_1\delta_1-\delta_1\delta_3=(a_1-\delta_3)\delta_1$. Hence $0<(a_1-\delta_3)\delta_1$ and as $\delta_1>0$, we deduce that $a_1>\delta_3$.  

Now assume that $\delta_3>\delta_1$. Since $a_1\delta_1+a_3\delta_3<m\le \max\{\delta_1,\delta_3\}$, if $a_1>\delta_3$, then $m\ge a_1\delta_1+a_3\delta_3>\delta_3\delta_1+a_3\delta_3=(\delta_1+a_3)\delta_3$.
We have $\delta_3\ge m>(\delta_1+a_3)\delta_3$, whence $1>\delta_1+a_3$, or equivalently, $0\geq \delta_1+a_3$ and so $a_3\le -\delta_1$.
 
\item This case is analogous. 

\item Remember that ${\bs x_m}=\bs \lambda_m$ implies $0< x_{m3}=a_3+b_3\le \delta_1$. So, if $a_3\le-\delta_1$, we have
$0< a_3+b_3\le -\delta_1+b_3$, and then $\delta_1< b_3$. If $\delta_1< b_3$, we obtain $a_3+\delta_1< a_3+b_3\le\delta_1$, and then $a_3<0$.
 
\item If ${\bs x_m}=\bs \lambda_m$, we have too that $-\delta_3<x_{m1}=a_1+b_1\le 0$ (Lemma \ref{negativo-l1}). If $a_1 > \delta_3$, then $0\ge a_1+b_1 > \delta_3+b_1$, and so $b_1 < -\delta_3$; while if $b_1 \le -\delta_3$, we have $-\delta_3< a_1+b_1 \le a_1-\delta_3$, and then $0< a_1$.
   
\item This case is similar as case (3)

\item The proof is analogous to case (4).\qedhere
\end{enumerate}
\end{proof}

Now, we are ready to proof the necessary result to finish the Theorem \ref{main}.  
\begin{lemma}\label{mainlemma}
 Consider, as in the proof of Theorem \ref{main} that
 \[
 \mathbf z-\tau_{\bs x_m}^-=a_1\mathbf v_1+a_3\mathbf v_3 \mbox{ and } \tau_{\bs x_m}^+-\mathbf z=b_1\mathbf v_1+b_3\mathbf v_3,
 \]
 with $\mathbf z=(z_1,z_2,z_3)\in\mathbb N^3$.
 \begin{enumerate}[(1)]
 \item If $(x_{m1},x_{m3})$ and $(a_1,a_3)$ are both $\lambda$-B\'ezout couples, then  $(b_1,b_3)$ is a B\'ezout couple. Similarly, if $(x_{m1},x_{m3})$ and $(b_1,b_3)$ are both $\mu$-B\'ezout couples, then $(a_1,a_3)$ is a B\'ezout couple.
 \item If ${\bs x_m}=\bs \lambda_m$, then $(a_1,a_3)$ is a $\lambda$-B\'ezout couple, and if ${\bs x_m}=\bs \mu_m$, then $(b_1,b_3)$ is a $\mu$-B\'ezout couple.
 \end{enumerate} 
\end{lemma}
\begin{proof}
 \begin{enumerate}[(1)]
  \item If both are $\lambda$-B\'ezout couples, we have by definition and Lemma \ref{negativo-l1}: 
  $-\delta_3  <       x_{m1}     \le    0$,                $      0      <       x_{m3}        \le    \delta_1$,  
  $     0    \le     -a_1          <    \delta_3$, and     $ -\delta_1  \le       -a_3        <   0$. Hence 
  $ -\delta_3   <   x_{m1}-a_1=b_1   <    \delta_3$ and     $ -\delta_1   <    x_{m3}-a_3=b_3    <    \delta_1$. 

  As $0<m= b_1\delta_1+b_3\delta_3\le \max\{\delta_1,\delta_3\}$, we deduce $b_1b_3\le 0$.
  This proves that $(b_1,b_3)$ is a B\'ezout couple.
   
  If both are $\mu$-B\'ezout couples, the proof is similar.

 \item Recall that by Lemma \ref{secondcoordinate} that, ${\bs x_m}=\bs \lambda_m$ if and only if $\tau_{m2}>0$. Hence $\mathbf z-\tau_{\bs x_m}^-=(z_1-\tau^-_{m1},z_2,z_3-\tau^+_{m3})$. So we have $-a_1r_{12}+a_3r_{32} = z_2 >0$. We distinguish two cases.
\begin{itemize}
\item $\delta_1>\delta_3$. We prove that $-\delta_1 < a_3 \le \delta_1$. Suppose to the contrary that  
   \begin{itemize}
   \item[(i)]   $a_3\le-\delta_1$, from Lemma \ref{technicallemma} (1) we have $a_1>\delta_3$ and so $z_2<0$, which is a contradiction;   or 
   
   \item[(ii)] $a_3 >\delta_1$, from Lemma \ref{technicallemma}
        \begin{itemize}
         \item[(2)] (sufficient condition) implies $a_1\le -\delta_3$,
         \item[(3)] (swapping $a$ and $b$) yields $b_3<0$, 
         \item[(4)] (swapping $a$ and $b$) forces $0<b_1$. 
        \end{itemize}
       As $\delta_1>\delta_3$, from Corollary \ref{performance}, $\tau_{m3}<0$. Then $\tau_{\bs x_m}^+-\mathbf z=(\tau^+_{m1}-z_1,\tau^+_{m2}-z_2,-z_3)$. So we have $-b_1r_{13}-b_3c_3 = -z_3<0$, and, as we are assuming $\delta_1>\delta_3$, we have $b_1\delta_1+b_3\delta_3<m\le \delta_1$. This implies $ (b_1-1)\delta_1<-b_3\delta_3<-b_3\delta_1$ and then $b_1-1<-b_3$, or equivalently, $b_1\le -b_3$. Since $c_3>r_{13}$, it follows that $b_1r_{13}+b_3c_3<0$ obtaining again a contradiction. 
   \end{itemize}
   
The case $\delta_3>\delta_1$ is analogous.
\end{itemize}
So, in this case $(a_1,a_3)$ is a B\'ezout couple. And as, $z_2>0$ for Lemma \ref{secondcoordinate} $(a_1,a_3)$ is a $\lambda$-B\'ezout couple. And from Lemma \ref{mainlemma} (1), we assure that $(b_1,b_3)$ is too a B\'ezout couple.

The case ${\bs x_m}=\bs \mu_m$ is completely similar as $\bs x_m=\bs \lambda_m$ case.\qedhere
\end{enumerate}
\end{proof}

\begin{example}
Take $S=\langle 8,41,79 \rangle$. We have $\delta_1=15-1-1=13$ and $\delta_3=4+5-3=6$. Following Theorem \ref{main}, we can compute a \emph{B\'ezout Table} for this couple.
\begin{table}[h]
\begin{center}
\begin{tabular}{|l|c|c||c||c|c|r|}
\hline 
Irr. & $\lambda_{i1}$ & $\lambda_{i3}$ & $i$ & $\mu_{i1}$ & $\mu_{i3}$ & Irr. \\
\hline \hline
$\checkmark$ & -5 & 11 & 1 & 1 & -2 & $\checkmark$ \\
\hline
$\checkmark$ & -4 & 9 & 2 & 2 & -4 & $\times$ \\
\hline
$\checkmark$ & -3 & 7 & 3 & 3 & -6 & $\times$ \\
\hline
$\checkmark$ & -2 & 5 & 4 & 4 & -8 & $\times$ \\
\hline
$\checkmark$ & -1 & 3 & 5 & 5 & -10 & $\times$ \\
\hline
$\checkmark$ & 0 & 1 & 6 & 6 & -12 & $\times$ \\
\hline
$\times$ & -5 & 12 & 7 & 1 & -1 & $\checkmark$ \\
\hline
$\times$ & -4 & 10 & 8 & 2 & -3 & $\times$ \\
\hline
$\times$ & -3 & 8 & 9 & 3 & -5 & $\times$ \\
\hline
$\times$ & -2 & 6 & 10 & 4 & -7 & $\times$ \\
\hline
$\times$ & -1 & 4 & 11 & 5 & -9 & $\times$ \\
\hline
$\times$ & 0 & 2 & 12 & 6 & -11 & $\times$ \\
\hline
$\times$ & -5 & 13 & 13 & 1 & 0 & $\checkmark$ \\
\hline
\end{tabular}
\end{center}
\caption{B\'ezout Table for $\delta_1=13$ and $\delta_3=6$,}
\end{table}

Thus, $\Delta(S)=\{1,2,3,4,5,6,7,13\}$.
\end{example}

Let $d=\max\{\delta_1,\delta_3\}$. Notice that this procedure has at least $d\log(d)$ complexity, and requires the precomputation of $\delta_1$ and $\delta_3$ . We will try to improve this in the next section. However, we can get some interesting theoretical consequences out of this (which was the initial motivation to write this manuscript).  By using Theorem \ref{main} we can prove two conjectures proposed by Malyshev \cite{BB}. Some partial solutions were provided in \cite{k-deltas}.

\begin{corollary}
Let $S$ be a nonsymmetric numerical semigroup with embedding dimension three and $|\Delta(S)| > 1$. If $1=\min \Delta(S)$, then $\{2,3\} \subseteq \Delta(S)$.
\end{corollary}
\begin{proof}
Suppose that $2 \not \in \Delta(S)$. By Theorem \ref{main}, $\bs\lambda_2,\bs\mu_2  \not \in I_{\delta_1,\delta_3}$. Hence (recall that $\bs \lambda_i=\bs \lambda_j+\bs \lambda_k$ implies $i=j+k$).  We must  then have $\bs\lambda_2=2\bs\lambda_1$ and $\bs\mu_2=2\bs\mu_1$. However, since $\bs\lambda_l+(\delta_3,-\delta_1)=\bs\mu_l$ for every $l$ (Lemma \ref{negativo-l1}), this implies \[\bs \lambda_1+\bs\mu_1= 2\bs\lambda_1+(\delta_3,-\delta_1)=\bs\lambda_2+(\delta_3,-\delta_1)=\bs\mu_2=2\bs \mu_1.\]
It follows that $\bs\mu_1=\bs\lambda_1$, contradicting the definition.

Now assume that  $3 \not \in \Delta(S)$. By Theorem \ref{main} $\bs\lambda_3,\bs\mu_3 \not \in I_{\delta_1,\delta_3}$. Hence we must have $\bs\lambda_3=\bs\lambda_1+\bs\lambda_2$ and $\bs\mu_3=\bs\mu_1+\bs\mu_2$. Here we obtain 
\[
\bs \lambda_1 +\bs \mu_2= \bs \lambda_1+\bs \lambda_2+(\delta_3,-\delta_1)= \bs\lambda_3+(\delta_3,-\delta_1)=\bs \mu_3=\bs\mu_1+\bs \mu_2,
\]
and thus $\bs\mu_1=\bs\lambda_1$, yielding again a contradiction.
\end{proof}

\begin{remark}
In the proof of the above corollary the contradiction is reached once we obtain that $\bs \lambda_i=\bs \lambda_j+\bs \lambda_k$ and $\bs \mu_i=\bs \mu_j+\bs \mu_k$ for some $i,j,k \in \mathbb{Z}^+$.

So we cannot guarantee that $4 \in \Delta(S)$ under the same assumptions, because in counterexamples such as $S=\langle 7,18,19 \rangle$ we get $\bs\mu_4=\bs\mu_1+\bs\mu_3$ and  $\bs\lambda_4=\bs\lambda_2+\bs\lambda_2$, whence $\bs\lambda_4 \not \in I_{\delta_1,\delta_3}$ and $\bs\mu_4 \not \in I_{\delta_1,\delta_3}$. So $4 \not \in \Delta(S)$.
\end{remark}

Under some extra assumptions we can get more information.

\begin{corollary}
If, in addition, $\min\{\delta_1,\delta_3\}=1$, then $\Delta(S)=\{1,2,\ldots,\max\{\delta_1,\delta_3\}\}$.
\end{corollary}
\begin{proof}
Suppose $\delta_1=1$. Then we have that $\delta_3-i=-i\delta_1+1\delta_3$ for $i\in\{0,\ldots,\delta_3-2\}$ and $1=1\delta_1+0\delta_3$. We observe that all this couples are B\'ezout couples and by Remark \ref{withone}, we obtain that they are all irreducible.
\end{proof}

Another important fact that should be highlighted is that $\Delta(S)$ does not depend on the generators of $S$, but on $\delta_1$ and $\delta_3$. So nonsymmetric numerical semigroups with embedding dimension three and with the same $\delta_1$ and $\delta_3$ will have the same Delta sets.

\section{Euclid's Algorithm and Delta sets}

Let $\delta_j$ and $\delta_k$ be integers with  $1<\delta_j<\delta_k$ and $\{j,k\}=\{1,3\}$. We will highlight the fact that $\bs x_i=(x_{ik},x_{ij})$ is a B\'ezout couple by explicitly saying that $\bs x_i$ is a B\'ezout couple for  $(\delta_j,\delta_k)$.  In this setting $0<i\le \delta_k$,  $i=x_{ik}\delta_k+x_{ij}\delta_j$, $0\le |x_{ik}|\le \delta_j$ and $0\le |x_{ij}|\le \delta_k$.

Denote by $\peb[x]$ the largest integer less than $x$.

\begin{lemma}\label{modulo}
Let $\bs x_i=(x_{ik},x_{ij})$ be a B\'ezout couple for $(\delta_k,\delta_j)$ with $i\le \delta_j$. Then 
\[
\bs x'_i=(\peb [\delta_k/\delta_j]x_{ik}+x_{ij},x_{ik})
\]
is a B\'ezout couple for $(\delta_j, \delta_k \bmod \delta_j)$. 
\end{lemma}
\begin{proof}
Notice that, for $i<\delta_j$,
\[
\left(\peb [\delta_k/\delta_j]x_{ik}+x_{ij}\right)\delta_j+x_{ik}(\delta_k\bmod \delta_j)=x_{ik}\delta_k+x_{ij}\delta_j=i.
\]
For $i=\delta_j$, we have $\bs x_{\delta_j}=(0,1)$ and $\bs x'_{\delta_j}=(1,0)$. We check that $\bs x'_i$ are indeed B\'ezout couples for $(\delta_j, \delta_k \bmod \delta_j)$, $i<\delta_j$. 

As $0<x_{ik}\delta_k+x_{ij}\delta_j=i< \delta_j$, dividing by $\delta_j$ we obtain 
\[
0<(\delta_k/\delta_j)x_{ik}+x_{ij}< 1.
\]
Since $0<x_{ik}\delta_k+x_{ij}\delta_j<\delta_j$ and $\delta_k=\peb[\delta_k/\delta_j]\delta_j+\delta_k\bmod \delta_j$, we have 
$0<x_{ik}(\peb[\delta_k/\delta_j]\delta_j+\delta_k\bmod \delta_j)+x_{ij}\delta_j<\delta_j$. Dividing again by $\delta_j$ we obtain
\[
0<(\peb[\delta_k/\delta_j]x_{ik}+x_{ij})+(x_{ik}/\delta_j)(\delta_k\bmod \delta_j)<1.
\]
We distinguish two cases depending on the sign of $x_{ik}$.
\begin{enumerate}[$\bullet$]
\item For $x_{ik}> 0$, as $\bs x_i$ is a B\'ezout couple, we know that $x_{ik}\le\delta_j$. Observe that
$\peb[\delta_k/\delta_j]x_{ik}+x_{ij}< (\delta_k/\delta_j)x_{ik}+x_{ij} < 1$. So $\peb [ \delta_k/\delta_j]x_{ik}+x_{ij}\le 0$, 
which gives us, using the second set of inequalities, and $x_{ik}\le \delta_j$:
\[
0\le -(\peb [\delta_k/\delta_j]x_{ik}+x_{ij}) <(x_{ik}/\delta_j)(\delta_k\bmod \delta_j)\le \delta_k\bmod \delta_j.
\]
We have obtained that both coordinates of $\bs x'_i$ satisfy $ -(\delta_k\bmod \delta_j)<\peb [\delta_k/\delta_j]x_{ik}+x_{ij}\le 0  $ and $0<x_{ik}\le \delta_j$.
\item While, if $x_{ik}\le 0$, as $\bs x_i$ is a B\'ezout couple, we have $x_{ik}>-\delta_j$. Also,  
$\peb[\delta_k/\delta_j]x_{ik}+x_{ij}\ge (\delta_k/\delta_j)x_{ik}+x_{ij} >0$. So, from the second set of inequalities we have $\peb[\delta_k/\delta_j]x_{ik}+x_{ij}<1-(x_{ik}/\delta_j)(\delta_k\bmod \delta_j)$, 
which yields
\[
0< \peb [\delta_k/\delta_j]x_{ik}+x_{ij} < 1+|(x_{ik}/\delta_j)|\delta_k\bmod \delta_j<1+\delta_k\bmod \delta_j.
\]
Deducing that $0<\peb [\delta_k/\delta_j]x_{ik}+x_{ij}\le\delta_k\bmod \delta_j$, and as $-\delta_j<x_{ik}\le 0$ we obtain again that $\bs x'_i$ is a B\'ezout couple for $(\delta_j, \delta_k \bmod \delta_j)$.\qedhere
\end{enumerate}
\end{proof}
\begin{remark}
From the proof of Lemma \ref{modulo}, it follows that 
\begin{itemize}
\item if $\bs x_i$ is a $\lambda$-B\'ezout couple, then $\bs x'_i$ is a $\mu$-B\'ezout couple;
\item if $\bs x_i$ is a $\mu$-B\'ezout couple, then $\bs x'_i$ is a $\lambda$-B\'ezout couple.
\end{itemize} 
\end{remark}

The above construction can be reversed.

\begin{lemma}\label{map}
 Let $i$ be a natural number $i\leq \delta_j$, and we consider $\bs x'_i=(x'_{ij},x'_{ik})$ a B\'ezout couple for $(\delta_j,\delta_k\bmod \delta_j)$. Then $\bs x_i=(x'_{ik},x'_{ij}-x'_{ik}\peb [\delta_k/\delta_j])$ is a B\'ezout couple for $(\delta_k,\delta_j)$.
\end{lemma}
\begin{proof}
It is clear that for $\bs x'_{\delta_j}=(1,0)$ we obtain $\bs x_{\delta_j}=(0,1)$. So, we can consider $i<\delta_j$. It is also easy  to check that 
$x'_{ik}\delta_k+(x'_{ij}-x'_{ik}\peb [\delta_k/\delta_j])\delta_j=x'_{ij}\delta_j+x'_{ik}(\delta_k-\peb [\delta_k/\delta_j]\delta_j)=x'_{ij}\delta_j+x'_{ik}(\delta_k\bmod \delta_j)=i$.

From this, we have $0<x'_{ij}\delta_j+x'_{ik}(\delta_k\bmod \delta_j)<\delta_j$. Dividing by $\delta_j$ we obtain:
\[
0<x'_{ij}+(x'_{ik}/\delta_j)(\delta_k\bmod \delta_j)<1.
\]
From here, and using that $\delta_k\bmod\delta_j=\delta_k-\peb[\delta_k/\delta_j]\delta_j$, we have $0<x'_{ij}+(x'_{ik}/\delta_j)(\delta_k-\peb[\delta_k/\delta_j]\delta_j)<1$, or equivalently 
\[
0<x'_{ij}-\peb[\delta_k/\delta_j]x'_{ik}+(x'_{ik}/\delta_j)\delta_k<1.
\]
As in Lemma \ref{modulo}, we distinguish two cases.
\begin{enumerate}[$\bullet$]
\item If $x'_{ik}\le 0$, we know that $-\delta_j<x'_{ik}$ and we have $(x'_{ik}/\delta_j)\delta_k\le 0$. Then $x'_{ij}-\peb[\delta_k/\delta_j]x'_{ik}>0$, because both summands are positive. So we have $0<x'_{ij}-\peb[\delta_k/\delta_j]x'_{ik}<1-(x'_{ik}/\delta_j)\delta_k<1+\delta_k$, since $-x'_{ik}=|x'_{ik}|<\delta_j$. Hence $x'_{ij}-\peb[\delta_k/\delta_j]x'_{ik}\le \delta_k$. Obtaining that $\bs x_i$ is a B\'ezout couple. 
\item If $x'_{ik}>0$,  we know that $x'_{ik}\le\delta_j$ and, from the first equation, we have that $x'_{ij}<1$, or equivalently $x'_{ij}\le 0$. So we can write, $0\le -x'_{ij}<(x'_{ik}/\delta_j)(\delta_k\bmod \delta_j)=(x'_{ik}/\delta_j)(\delta_k-\peb[\delta_k/\delta_j]\delta_j)=(x'_{ik}/\delta_j)\delta_k-\peb[\delta_k/\delta_j]x'_{ik}$. Adding $ \peb[\delta_k/\delta_j]x'_{ik}$, we obtain $0<\peb[\delta_k/\delta_j]x'_{ik}\le\peb[\delta_k/\delta_j]x'_{ik}-x'_{ij}<(x'_{ik}/\delta_j)\delta_k<\delta_k$, as $x'_{ik}<\delta_j$. So, $-\delta_k<x'_{ij}-\peb[\delta_k/\delta_j]x'_{ik}\le 0$ and as $0<x'_{ik}\le \delta_j$. This proves that $\bs x_i$ is a B\'ezout couple.\qedhere
\end{enumerate}
\end{proof}

\begin{defn}
Set $\mathcal B(\delta_k,\delta_j)=\{\bs x_i \mid \bs x_i\mbox{ is a B\'ezout couple for } (\delta_k,\delta_j)\}$.
\end{defn}

\begin{proposition}
The map $f\colon\mathcal B(\delta_j,\delta_k\bmod \delta_j)\rightarrow\mathcal B(\delta_k,\delta_j)$,  $\bs x'_i\mapsto \bs x_i$, coming from Lemma \ref{map}, is additive and injective.  
\end{proposition}
\begin{proof}
Follows easily from the definition of $\bs x_i$ in Lemma \ref{map} and Lemma \ref{modulo}.
\end{proof}

\begin{remark}
Note that the elements in $\mathrm{Im}(f)$ correspond with B\'ezout couples for numbers smaller than or equal than $\delta_j$.
\end{remark}

Now, we want to prove that the irreducible B\'ezout couples in $\mathcal B(\delta_k,\delta_j)\setminus \mathrm{Im}(f)$ are  those with $x_{ik}=1$. 
First of all, we compute these B\'ezout couples in the next proposition.
\begin{proposition}\label{Bezoutwithone}
Let $\bs x_i=(1,x_{ij}) \in\mathcal B(\delta_k,\delta_j)$. Then $\bs x_i\not\in  \mathrm{Im}(f)$ if and only if  $-\peb[\delta_k/\delta_j]<x_{ij}\le 0$.
\end{proposition}
\begin{proof}
Take $\bs x_i=(1,x_{ij})\in\mathcal B(\delta_k,\delta_j)$. Hence $\delta_k+ x_{ij}\delta_j=i\le \delta_k$, and $x_{ij}\le 0$. 

If $x_{ij}> - \peb[\delta_k/\delta_j]$, then  $0\ge x_{ij}\ge - \peb[\delta_k/\delta_j]+1$. Hence $i\ge \delta_k-\peb [\delta_k/\delta_j]\delta_j+\delta_j=(\delta_k\bmod \delta_j)+\delta_j>\delta_j$. Therefore $\bs x_i\notin \mathrm{Im}(f)$. 

It is clear that $\delta_k+x_{ij}\delta_j$ with $x_{ij}\le -\peb[\delta_k/\delta_j]$ are elements smaller than $\delta_j$ and by Lemma \ref{modulo} the corresponding $\bs x_i$ is  in $\textrm{Im}(f)$. 
\end{proof}

\begin{remark}\label{lessthanj}
Observe that the case $x_{ik}=1$ and $x_{ij}=-\peb[\delta_k/\delta_j]$, yields an irreducible couple with $x_{ik}\delta_k+x_{ij}\delta_j=\delta_k-\peb [\delta_k/\delta_j]\delta_j=\delta_k\bmod \delta_j< \delta_j$.
\end{remark}

\begin{proposition}
Irreducible elements in $\mathcal B(\delta_k,\delta_j)\setminus \mathrm{Im}(f)$ are those B\'ezout couples with $x_{ik}=1$ and $-\peb[\delta_k/\delta_j]<|x_{ij}|\leq 0$.
\end{proposition}
\begin{proof}
From Remark \ref{withone} we know that B\'ezout couples of the form $(1,x_{ij})$ are irreducible. We need to prove that all elements bigger than $\delta_j$ different from these have associated reducible B\'ezout couples.

For this, suppose that $\delta_k>i>\delta_j$. Now, we can find  $q\le \peb[\delta_k/\delta_j]$ such that $\delta_k-q\delta_j>i>\delta_k-(q+1)\delta_j>0$; in particular, $i=\delta_k-(q+1)\delta_j+r$ with $r$ a positive integer such that $r<\delta_j$. From Remark \ref{withone}, Proposition \ref{Bezoutwithone} and Remark \ref{lessthanj}, we have that $\bs x_{i'}=(1,-(q+1))$ is an irreducible B\'ezout couple for some $i'<i$. 

If $x_{ik}>1$, consider the B\'ezout couple $\bs x_r=(x_{rk},x_{rj})$ associated to $r$ with $0<x_{rk}\le \delta_j$ (and $-\delta_k<x_{rj}\le 0$; Lemma \ref{negativo-l1}). If $x_{rk}=\delta_j$, then $r=\delta_j\delta_k +x_{rj}\delta_j=\delta_j(\delta_k+x_{rj}) \ge \delta_j$, contradicting that $r<\delta_j$. Hence $x_{rk}<\delta_j$ and $\bs x_i=(1,-(q+1))+\bs x_r$, obtaining that $\bs x_i$ is not irreducible. 

If $x_{ik}<0$, write $i=\delta_j+r$. We consider now the B\'ezout couple $\bs x_r$ associated to $r$ with $-\delta_j<x_{rk}\le 0$ and $0<x_{rj}\le \delta_k$. Then $\bs x_i=(0,1)+(x_{rk},x_{rj})$, if $x_{rj}+1\le \delta_k$, or equivalently, $x_{rj}\neq \delta_k$. If $x_{rj}=\delta_k$, then $r=x_{rk}\delta_k+\delta_k\delta_j= \delta_k(x_{rk}+\delta_j)\ge \delta_k$, a contradiction. 

If $x_{ik}=0$ then $i=t\delta_j$. Then, if $t=1$ we have $i=\delta_j\in \mathrm{Im} (f)$ and for $t>1$ we deduce $(0,t)=(0,t-1)+(0,1)$ obtaining again that $(x_{ik},x_{ij})$ is a reducible Bezout couple. 
\end{proof}

\begin{proposition}
Irreducible B\'ezout couples of $\mathcal B(\delta_k,\delta_j)\cap \mathrm{Im}(f)$ are only those that come from irreducible B\'ezout couples  in $\mathcal B(\delta_j,\delta_k\bmod \delta_j)$.
\end{proposition}
\begin{proof}
Irreducible elements in $\mathrm{Im}(f)$ are those elements in $\mathrm{Im}(f)$ that can not be written as sum of two elements in $\mathcal B(\delta_k,\delta_j)$. So additivity of $f$ ensures that the pre-images of these irreducible elements can not be expressed as sum of elements in $\mathcal B(\delta_j,\delta_k\bmod \delta_j)$.

If we have $\bs x'_i$ an irreducible element in $\mathcal B(\delta_j,\delta_k\bmod \delta_j)$, we know that $i<\delta_j$ so we can not to write $\bs x_i$ as sum of elements out of $\mathrm{Im}(f)$, because elements out of $\mathrm{Im}(f)$ correspond with numbers bigger than $\delta_j$. \qedhere
\end{proof}

Observe that according to the last two results and Theorem \ref{main} the elements in $\Delta(S)$ can be obtained in the following way.
\begin{itemize}
\item First compute the couples $(1,-t)$ with $t\in \{0,\ldots, \peb [\delta_3/\delta_1]\}$. These correspond to the values $\delta_3-t\delta_1$ in $\Delta(S)$. 

\item Then proceed in the same way with $(\delta_3,\delta_3\bmod \delta_1)$, until we reach $\gcd(\delta_1,\delta_3)$. 
\end{itemize}
Observe that if $\delta_1<\delta_3$, then $\peb[\delta_1/\delta_3]=0$, and we go directly to the second step, swapping $\delta_1$ with $\delta_3$.

Thus the possible values in $\Delta(S)$ are those arising in the calculation of $\gcd(\delta_1,\delta_3)$ in the naive way. That is $\gcd(a,b)=\gcd(a,b-a)$ while $b>a$, and then we swap positions and start anew, until we reach 0 in the second argument. The output is just the set of $a$'s and $b$'s appearing in the process removing 0.
 
\begin{example}
Let $S=\langle 1407,26962,35413\rangle$. Its minimal presentation is: 
\[
\{((411,0,0),(0,7,11));((0,91,0),(284,0,58));((0,0,69),(127,84,0))\},
\]
so $\delta_1=393$, $\delta_2=251$, $\delta_3=142$. We start the Euclid's algorithm with $\delta_k=393$ and $\delta_j=142$ to obtain:

\[
\begin{array}{cccccl}
\delta_k=393 & \delta_j=142 & & &\\
 & _{(1,0)} & _{(1,-1)} & {\color{red}_{(1,-2)}} & & \\
 & 393 & 251 & {\color{red}109} & {\color{red}=\delta_k\bmod \delta_j} &\\
 \delta_k=142 & \delta_j=109 & & & &\\ 
 & _{(0,1)} & {\color{red}_{(-1,3)}}  &  & & \\
 & 142 & {\color{red}33} & {\color{red}=\delta_k\bmod \delta_j}   & & \\
\delta_k=109 & \delta_j=33 & & & & \\
 & _{(1,-2)} & _{(2,-5)} & _{(3,-8)} & {\color{red}_{(4,-11)}} &   \\
 & 109 & 76 & 43 & {\color{red}10} &  {\color{red}=\delta_k\bmod \delta_j}\\
\delta_k=33 & \delta_j=10 & & & & \\
 &_{(-1,3)} & _{(-5,14)} & _{(-9,25)} & {\color{red}_{(-13,36)}} &  \\
 & 33 & 23 & 13 & {\color{red}3}  & {\color{red}=\delta_k\bmod \delta_j} \\
\delta_k=10 & \delta_j=3 & & & &  \\
 & _{(4,-11)} & _{(17,-47)} & _{(30,-83)} & {\color{red}_{(43,-119)}} &   \\
 & 10 & 7 & 4 & {\color{red}1} &  {\color{red}=\delta_k\bmod \delta_j}\\ 
\delta_k=3 & \delta_j=1 & & & &  \\
 & _{(-13,36)}  & _{(-56,155)} & _{(-99,274)} & {\color{red}_{(-142,393)}} &  \\
 & 3 & 2 & 1 & {\color{red}0}  &{\color{red}=\delta_k\bmod \delta_j}
\end{array}
\]
The coordinates shown over the integers are coordinates with respect to $\delta_k=393$ and $\delta_j=142$.
 So, we have that the $\mu$-B\'ezout couples are:
\[
\{(1,0),(1,-1),(1,-2),(2,-5),(3,-8),(4,-11),(17,-47),(30,-83),(43,-119)\}
\]
and the $\lambda$-B\'ezout couples are:
\[
\{(0,1),(-1,3),(-5,14),(-9,25),(-13,36),(-56,155),(-99,274)\}.
\]

Observe that the couples $(43,-119)$ and $(-99,274)$ correspond to $i=1$, and both are irreducible B\'ezout couples, but we only need to compute them once, so we can ``forget'' $(43,-119)$ when we are looking for the Delta set.

Thus the Delta set for $S$ is:
\[
\Delta(S)=\{1,2,3,4,7,10,13,23,33,43,76,109,142,251,393\}.
\]

In practice, when we are only interested in the Delta set, we do not need to keep track of the B\'ezout couples, just the integers appearing in the greatest common divisor computation.

\begin{verbatim}
gap> deltasetnsembdim3(1407, 26962, 35413);time;
[ 1, 2, 3, 4, 7, 10, 13, 23, 33, 43, 76, 109, 142, 251, 393 ]
1
\end{verbatim}
The time is in milliseconds, that is, it takes 1 millisecond to compute $\Delta(S)$. The current procedure \texttt{DeltaSetOfNumericalSemigroup} in \texttt{nume\-rical\-sgps} executed with this example was stopped after one day without an output. The  implementation of \texttt{DeltaSetOfNumericalSemigroup} is based on a dynamical procedure presented in \cite{MOP} and was kindly programmed by Chris O'Neil (see the contributions section in the manual of \texttt{numericalsgps}). Of course it was meant for arbitrary numerical semigroups, and not just nonsymmetric numerical semigroups with embedding dimension three. The idea in \cite{MOP} is to compute all Delta sets of elements up to when this calculation becomes periodical. In our example the bound for periodicity is just too big; this is why it was not able to give an answer after one day of computation. 
\begin{verbatim}
gap> DeltaSetPeriodicityBoundForNumericalSemigroup(
> NumericalSemigroup(1407, 26962, 35413));
916982754
\end{verbatim}

Next we show another examples of execution with their timings (milliseconds).
\begin{verbatim}
gap> s:=NumericalSemigroup(101,301,510);; 
gap> DeltaSetOfNumericalSemigroup(s);time;
[ 1, 2, 3, 4, 5, 6, 7, 8, 9, 10, 11, 12, 13, 14, 15, 16, 17, 18 ]
10271
gap> deltasetnsembdim3(101,301,510);time;
[ 1, 2, 3, 4, 5, 6, 7, 8, 9, 10, 11, 12, 13, 14, 15, 16, 17, 18 ]
1
\end{verbatim}

\begin{verbatim}
gap> s:=NumericalSemigroup(151,301,510);;
gap> DeltaSetOfNumericalSemigroup(s);time;
[ 1, 2, 3, 5, 7, 12, 17, 22 ]
4976
gap> deltasetnsembdim3(151,301,510);time;
[ 1, 2, 3, 5, 7, 12, 17, 22 ]
1
\end{verbatim}

\begin{verbatim}
gap> deltasetnsembdim3(8,41,79);time;      
[ 1, 2, 3, 4, 5, 6, 7, 13 ]
1
\end{verbatim}
\end{example}

In \cite[Section 2]{jmda} we present a procedure to compute the primitive elements of $\ker\varphi$, or equivalently, a Graver basis of $M$, that is, the set of minimal nonzero elements of $M$ with respect to $\sqsubseteq$, defined as $(x_1,x_2,x_3) \sqsubseteq (y_1,y_2,y_3)$ if $x_iy_i\ge 0$ and $|x_i|\le |y_i|$ for $i\in\{1,2,3\}$. It is easy to prove that $\mathbf v_1$, $\mathbf v_2$ and $\mathbf v_3$ are elements in this Graver basis. So for instance, in order to compute $\mathbf v_1$, we look for the elements in the Graver basis of $M$ of the form $(a,b,c)$ with $a\neq 0$ and $b c\ge 0$, and choose the element with least $|a|$ (which corresponds with $c_1$).  In this way we can compute a minimal presentation for $S$, and consequently $\delta_1$ and $\delta_3$. The algorithm presented in \cite{jmda} has the same complexity as Euclid's greatest common divisor algorithm. The timings presented in the above examples for \texttt{deltasetembdim3} include the calculation of a minimal presentation.

Checking whether or not $S$ is nonsymmetric can be done easily by using \cite[Theorem 10.6]{RG}, which relies also in greatest common divisor calculations. The semigroup $S$ is symmetric if and only if it is of the form $\langle am_1,am_2, bm_1+cm_2\rangle$ with 
\begin{itemize}
\item $m_1$ and $m_2$ coprime integers greater than one,
\item $a$, $b$ and $c$ nonnegative integers with $a\ge 2$, $b+c\ge 2$ and $\gcd(a,bm_1+cm_2)=1$.
\end{itemize}
So if we want to check whether or not  $S=\langle n_1,n_2,n_3\rangle$ is symmetric, we take all possible partitions $\{\{n_i,n_j\},\{n_k\}\}$ with $\{i,j,k\}=\{1,2,3\}$. Then for each partition we compute $a=\gcd\{n_i,n_j\}$, and if it is greater than one, we check if $n_k\in \langle n_i/a,n_j/a\rangle \setminus\{n_i/a,n_j/a\}$. If so, the semigroup is symmetric. If it is not the case for any partition, then $S$ is not symmetric.

\begin{example}
Let $S=\langle 4,6,9\rangle$. Then $\gcd(4,6)=2$ and $9\in \langle 2,3\rangle\setminus \{2,3\}$. Whence $S$ is symmetric.

For $S=\langle 3,5,7\rangle$, every two generators are coprime  (and so $a=1$), whence  $S$ is not symmetric. If we compute the Graver basis for $M\equiv 3x+5y+7z=0$ using \cite{jmda}, we obtain
\[
G=\{ ( 0, -7, 5 ), ( 1, -2, 1 ), ( 1, 5, -4 ), ( 2, 3, -3 ), ( 3, 1, -2 ), 
  ( 4, -1, -1 ), ( 5, -3, 0 ), ( 7, 0, -3 ) \}
\]
(we remove $-\mathbf v$ if $\mathbf v$ is already in the basis). When looking for $\mathbf v_1$ we need to search for the elements $(a,b,c)\in G$ with $a\neq 0$ and $bc\ge 0$: $\{(4,-1,-1), (5,-3,0), (7,0,-3)\}$. Then choose the element with minimal $|a|$. In this case, $\mathbf v_1=(4,-1,-1)$, which yields $((4,0,0),(0,1,1))\in \ker \phi$. We proceed in the same way with $\mathbf v_2$ and $\mathbf v_3$.
It follows that a minimal presentation is 
\[
\{ ((4,0,0),(0,1,1)), ((0,2,0),(1,0,1)), ((0,0,2),(3,1,0))\}.
\] 
Hence $\delta_1=2=\delta_3$ and $\delta_2=0$. In this case $\Delta(S)=\{2\}$.
\end{example}

\end{document}